\documentclass[11pt,reqno]{amsart}
 \usepackage{amssymb,amsfonts,amscd,amsbsy}
\usepackage[mathscr]{eucal}
\usepackage{url}

\newtheorem{thm}{Theorem}[section]

\newtheorem{prop}[thm]{Proposition}

\theoremstyle{definition}

\numberwithin{equation}{section}

\makeatletter
\def\imod#1{\allowbreak\mkern5mu({\operator@font mod}\,\,#1)}
\makeatother

\begin{document}

\title[Mixed mock modular $q$-series]
{Mixed mock modular $q$-series} 
 
\author{Jeremy Lovejoy and Robert Osburn}

\address{CNRS, LIAFA, Universit{\'e} Denis Diderot - Paris 7, Case 7014, 75205 Paris Cedex 13, FRANCE}

\address{School of Mathematical Sciences, University College Dublin, Belfield, Dublin 4, Ireland}

\address{IH{\'E}S, Le Bois-Marie, 35, route de Chartres, F-91440 Bures-sur-Yvette, FRANCE}

\email{lovejoy@liafa.jussieu.fr}

\email{robert.osburn@ucd.ie, osburn@ihes.fr}


\date{\today}

\begin{abstract} 
Mixed mock modular forms are functions which lie in the tensor space of mock modular forms and modular forms.  As $q$-hypergeometric series, mixed mock modular forms appear to be much more common than mock theta functions.  In this survey we discuss some of the ways such series arise.
\end{abstract}

\maketitle

\section{Introduction}

The mock theta functions are one of Ramanujan's greatest discoveries.   After eighty years of mystery, the last decade has seen these functions finally understood in the context of modular forms.   Mock theta functions are now known to be holomorphic parts of certain weight $1/2$ harmonic Maass forms \cite{On1,Za1,Zw1}.  More generally, the holomorphic part of a weight $k$ harmonic Maass form is called a mock modular form of weight $k$.  If we allow multiplication of a mock modular form by a modular form, then we have a \emph{mixed} mock modular form.   

Mock modular forms in algebra, number theory, and physics are often of the mixed variety.  For example, mixed mock modular forms have recently appeared as characters in the theory of affine Lie superalgebras \cite{ac,BO1}, as generating functions for exact formulas for the Euler numbers of certain moduli spaces \cite{bm}, for Joyce invariants \cite{Me-Ok1} and for linking numbers in 3-manifolds \cite{fm}, in the quantum theory of black holes and wall-crossing phenomenon \cite{dmz}, in relation to other automorphic objects \cite{bk,cr} and in the combinatorial setting of $q$-series and partitions (e.g. \cite{andrewsshort,andrewsrep,andrewsearly,Antube,arz,bhmv,bma,rhoades}).   

As $q$-series, mixed mock modular forms appear to be much more common than mock theta functions.  In this paper we briefly survey some of the ways such series arise.  We begin in Section 2 by recalling two building blocks of mixed mock modular forms, namely Appell functions and indefinite theta functions. In Section 3, we discuss how $q$-series transformations, Bailey pairs, the Bailey chain, and partial theta identities naturally lead to mixed mock modular $q$-series. For example, the multisum

\begin{equation} \label{Bk}
\mathcal{B}^{(k)}(q):= \sum_{n_k \geq n_{k-1} \geq \cdots \geq n_1 \geq 0} \frac{q^{n_k^2+n_{k-1}^2+\cdots+n_1^2}}{(q)_{n_k-n_{k-1}}\cdots(q)_{n_2-n_1}(-q)_{n_1}^2}
\end{equation}

\noindent is a mixed mock modular form for $k \geq 2$ \cite{lo1}. Here we have employed the usual $q$-series notation,
\begin{equation*}
(a_1, a_2, \dotsc, a_j)_n = (a_1, a_2, \dotsc, a_j ; q)_n := \prod_{k=1}^{n}(1-a_1 q^{k-1})(1 - a_2 q^{k-1}) \cdots (1-a_j q^{k-1}),
\end{equation*} 

\noindent valid for $n \in \mathbb{N} \cup \{\infty\}$. 

In Section 4, we give an example of the type of identity one can prove for mixed mock modular forms.  To state this identity, recall that on page 9 of the lost notebook \cite{Ra1}, Ramanujan recorded what are now known as the tenth order mock theta functions.   Two of these, $\chi$ and $X$, are defined by

\begin{equation*} \label{chi}
\chi(q):= \sum_{n \geq 0} \frac{(-1)^n q^{(n+1)^2}}{(-q)_{2n+1}}
\end{equation*}

\noindent and

\begin{equation*} \label{X}
X(q):= \sum_{n \geq 0} \frac{(-1)^n q^{n^2}}{(-q)_{2n}}.
\end{equation*}

\begin{thm} \label{main}
We have
\begin{equation} \label{b1chix}
\mathcal{B}^{(2)}(q) + \frac{2}{(q^2, q^3; q^5)_{\infty}} \chi(q) - \frac{2}{(q, q^4; q^5)_{\infty}} X(q) = -\frac{(q)_{\infty}}{(-q)_{\infty}^2}.
\end{equation}
\end{thm}
We close the paper in Section 5 with some suggestions for future study. 

\section{Level $\ell$ Appell functions and indefinite theta series}
We briefly recall Appell functions and indefinite theta functions, each of which is a building block of mixed mock modular forms.   The level $\ell$ Appell function $A_{\ell}(a,b,q)$ is defined by \cite{Zw2}

\begin{equation} \label{Appelldef}
A_{\ell}(a,b,q) := a^{\ell/2}\sum_{n \in \mathbb{Z}} \frac{(-1)^{\ell n}q^{\ell n(n+1)/2}b^n}{1-aq^n}
\end{equation}
and the indefinite theta series $f_{a,b,c}(x,y,z)$ is given by \cite{Hi-Mo1}

\begin{equation} \label{fdef}
f_{a,b,c}(x,y,q) := \left(\sum_{r,s \geq 0} - \sum_{r,s < 0}\right) (-1)^{r+s} x^r y^s q^{a \binom{r}{2} + brs + c \binom{s}{2}}.
\end{equation}

If $f$ is a modular form, the functions $\frac{1}{f} A_{\ell}(a,b,q)$ and $\frac{1}{f} f_{a,b,c}(x,y,q)$ are in general mixed mock modular forms \cite{ac,Hi-Mo1,Zw1,Zw2}.   In very special cases they may be mock modular (or even modular), one such case being the Appell-Lerch series $m(x,q,z)$, where 

\begin{equation} \label{Appell-Lerchdef} 
m(x,q,z) := \frac{1}{j(z,q)} \sum_{r \in \mathbb{Z}} \frac{(-1)^r q^{\binom{r}{2}} z^r}{1-q^{r-1} xz}.
\end{equation}   

\noindent Here $j(x,q):=(q,x,q/x)_{\infty}$.  For future reference we note that from the definition of $j(x, q)$, we have

\begin{equation} \label{j1}
j(q^{n} x, q) = (-1)^{n} q^{-\binom{n}{2}} x^{-n} j(x,q)
\end{equation}

\noindent where $n \in \mathbb{Z}$ and

\begin{equation} \label{j2}
j(x,q) = j(q/x, q) = -x j(x^{-1}, q).
\end{equation}

Hickerson and Mortenson \cite{Hi-Mo1} recently gave an explicit expression for the indefinite theta series \eqref{fdef} in terms of Appell-Lerch series \eqref{Appell-Lerchdef}.  We cite a special case of their result.  Define

\begin{equation} \label{g}
\begin{aligned}
g_{a,b,c}(x, y, q, z_1, z_0) & := \sum_{t=0}^{a-1} (-y)^t q^{c\binom{t}{2}} j(q^{bt} x, q^a) m\left(-q^{a \binom{b+1}{2} - c \binom{a+1}{2} - t(b^2 - ac)} \frac{(-y)^a}{(-x)^b}, q^{a(b^2 - ac)}, z_0) \right)\\
& + \sum_{t=0}^{c-1} (-x)^t q^{a \binom{t}{2}} j(q^{bt} y, q^c) m\left(-q^{c\binom{b+1}{2} - a\binom{c+1}{2} - t(b^2 -ac)}  \frac{(-x)^c}{(-y)^b}, q^{c(b^2 - ac)}, z_1\right)
\end{aligned}
\end{equation}
and
\begin{equation*}
\begin{aligned}
& \theta_{n,p}(x,y,q) := \frac{1}{\overline{J}_{0, np(2n+p)}} \sum_{r^{*} = 0}^{p-1} \sum_{s^{*}=0}^{p-1} q^{n\binom{r-(n-1)/2}{2} + (n+p)(r - (n-1)/2)(s+ (n+1)/2) + n \binom{s + (n+1)/2}{2}} \\
& \times \frac{(-x)^{r - (n-1)/2} (-y)^{s + (n+1)/2} J_{p^{2} (2n+p)}^{3} j(-q^{np(s-r)} x^{n} / y^{n}, q^{np^2}) j(q^{p(2n+p)(r+s) + p(n+p)} x^{p} y^{p}, q^{p^2 (2n + p)})}{j(q^{p(2n+p)r + p(n+p)/2} (-y)^{n+p} / (-x)^{n}, q^{p^{2} (2n+p)}) j(q^{p(2n+p)s + p(n+p)/2} (-x)^{n+p} / (-y)^{n}, q^{p^{2} (2n+p)})},
\end{aligned}
\end{equation*}

\noindent where $r := r^{*} + \{(n-1)/2 \}$ and $s:= s^{*} + \{ (n-1)/2 \}$ with $0 \leq \{ \alpha \} < 1$ denoting the fractional part of $\alpha$.  Also, $J_{m}:= J_{m, 3m}$ with $J_{a,m} := j(q^{a}, q^{m})$, and $\overline{J}_{a,m}:=j(-q^{a}, q^{m})$. Following \cite{Hi-Mo1}, we use the term ``generic" to mean that the parameters do not cause poles in the Appell-Lerch series or in the quotients of theta functions.   Let $n$ and $p$ be positive integers with $(n$, $p)=1$. For generic $x$, $y \in \mathbb{C}^{*}$, we have (see Theorem 0.3 in \cite{Hi-Mo1})

\begin{equation} \label{hm}
f_{n, n+p, n}(x, y, q) = g_{n, n+p, n}(x, y, q, -1, -1) + \theta_{n,p}(x,y,q).
\end{equation}
Note that since $m(x,q,z)$ is generally mock modular and $j(x,q)$ is modular, we explicitly see that such indefinite theta series are in general mixed mock modular forms.

\section{Some sources of $q$-hypergeometric mixed mock modular forms}
\subsection{$q$-series transformations}
We begin with a simple observation.  The typical two-term $q$-series transformation takes the form $F=fG$, where $F$ and $G$ are series and $f$ is an infinite product which, under standard specializations, is a modular form.  Thus any mock theta function to which such a transformation (non-trivially) applies immediately corresponds to a mixed mock modular form.  We give some examples.

First, consider a $_3\phi_2$ transformation \cite[Appendix III, Equation (III.10)]{Ga-Ra1}

\begin{equation} \label{3phi2}
\sum_{n\geq 0}\frac{\left(aq/bc, d, e\right)_n}{\left(q, aq/b, aq/c\right)_n}\left(\frac{aq}{de}\right)^n
=\frac{\left(aq/d, aq/e, aq/bc\right)_\infty}{\left(aq/b, aq/c, aq/de\right)_\infty}
\sum_{n\geq 0}\frac{\left(aq/de, b, c\right)_n}{\left(q, aq/d, aq/e\right)_n}\left(\frac{aq}{bc}\right)^n.
\end{equation}      

\noindent The series on the right-hand side of \eqref{3phi2} specializes in many different ways to give mock theta functions \cite{Al-Br-Lo1, Br-Lo1, Br-On1, Br-On-Rh1,Ka1}. In all of these cases the left-hand side of \eqref{3phi2} is then a mixed mock modular form.   To give an example, we let $a=1$, $d=1/e = x$ and $b,c \to \infty$ in \eqref{3phi2} to obtain

\begin{equation} \label{S}    
\mathcal{S}(x,q) := \sum_{n \geq 0} \frac{(x,1/x)_nq^n}{(q)_n} =  \frac{(xq,q/x)_{\infty}}{(q)_{\infty}}\sum_{n \geq 0} \frac{q^{n^2}}{(xq,q/x)_n}.
\end{equation}

\noindent Now for any root of unity $x \neq 1$ the series on the right-hand side of \eqref{S} is a mock theta function \cite{Br-On1}, so for such $x$ the series $\mathcal{S}(x,q)$ will in general be a mixed mock modular form.  Of note is the case 

\begin{equation} \label{Sspec}
\mathcal{S}(-e^{2\pi i/3},q) = 1 + \sum_{n \geq 1} \frac{(-q^3;q^3)_{n-1}q^n}{(-q)_{n-1}(q)_n} = \frac{(-q^3;q^3)_{\infty}}{(q^2;q^2)_{\infty}} \chi(q)
\end{equation}

\noindent where 

\begin{equation*}
\chi(q) := \sum_{n \geq 0} \frac{q^{n^2}(-q)_n}{(-q^3;q^3)_n}  
\end{equation*}
is a third order mock theta function \cite{Wat1} (not to be confused with the tenth order $\chi(q)$).   The mixed mock modular form in \eqref{Sspec} is the generating function for partitions without sequences \cite{andrewsshort,bma}. 

For another application of \eqref{3phi2}, let $a=q$, $d=y$, $e=q/y$, and $b,c \to \infty$ to get
\begin{equation} \label{Gleissberg}
(q)_{\infty}\sum_{n \geq 0} \frac{(y,q/y)_nq^n}{(q)_n} = \frac{(y,q/y)_{\infty}}{(q)_{\infty}}\sum_{n \geq 0} \frac{q^{n^2+n}}{(y,q/y)_{n+1}}. 
\end{equation}
It was explained in \cite{Br-Ma1} how the case $q=q^d$ and $y=-q^r$ of the left-hand side of \eqref{Gleissberg} is the generating function for certain partitions related to Gleissberg's extension of Schur's theorem \cite{Gl1}.  The sum on the right-hand side of \eqref{Gleissberg} is a so-called \emph{universal} mock theta function \cite{Go-Mc1,Hi1,Ka1}.   Thus the generating function for the Gleissberg-type partitions is a mixed mock modular form.  (See \cite{Br-Ma1} for more on this.)

Next consider two transformations due to Bailey (see \cite{AnlostIV} or \cite{bailey}),

\begin{equation} \label{Baileytrans}
\begin{aligned}
\sum_{n \geq 0} \frac{(\rho_1,\rho_2)_n(aq/f;q^2)_n \left(\frac{aq}{\rho_1\rho_2}\right)^n}{(q,aq/f)_n(aq;q^2)_n} &= \\ \frac{(aq/\rho_1,aq/\rho_2)_{\infty}}{(aq,aq/\rho_1\rho_2)_{\infty}}&\sum_{n \geq 0}\frac{(1-aq^{4n})(a,f;q^2)_n(\rho_1,\rho_2)_{2n}\left(\frac{a^3}{\rho_1^2\rho_2^2f}\right)^nq^{2n^2+2n}}{(1-a)(q^2,aq^2/f;q^2)_n(aq/\rho_1,aq/\rho_2)_{2n}}
\end{aligned}
\end{equation}
and
\begin{equation} \label{Baileytransbis}
\begin{aligned}
\sum_{n \geq 0} \frac{(r_1,r_2;q^2)_n(-aq/b)_{2n}\left(\frac{a^2q^2}{r_1r_2}\right)^n}{(q^2,a^2q^2/b^2;q^2)_n(-aq)_{2n}} &= \\ \frac{(a^2q^2/r_1,a^2q^2/r_2;q^2)_{\infty}}{(a^2q^2,a^2q^2/r_1r_2;q^2)_{\infty}}&\sum_{n \geq 0}\frac{(1-aq^{2n})(a,b)_n(r_1,r_2;q^2)_n\left(\frac{a^3}{br_1r_2}\right)^nq^{n^2+2n}}{(1-a)(q,aq/b)_n(a^2q^2/r_1,a^2q^2/r_2;q^2)_n}.
\end{aligned}
\end{equation}
\noindent There are several ways to turn the series on the right-hand side of \eqref{Baileytrans} or \eqref{Baileytransbis} into an Appell function \eqref{Appelldef}.  For example, setting $a=1$, $\rho_1=\rho_2 = -1$ and letting $f \to \infty$ in \eqref{Baileytrans} gives 

\begin{equation} \label{T1}
\mathcal{T}_1(q) := \sum_{n \geq 0} \frac{(-1)_n^2q^n}{(q)_n(q;q^2)_n} = \frac{2(-q)_{\infty}^2}{(q)_{\infty}^2}\sum_{n \in \mathbb{Z}} \frac{(-1)^nq^{3n^2+n}}{1+q^{2n}}.
\end{equation}

\noindent Comparing \eqref{T1} with Watson's expression for Ramanujan's third order mock theta function $f(q)$ \cite{Wat1},

\begin{equation} \label{fofq}
f(q) := \sum_{n \geq 0} \frac{q^{n^2}}{(-q)_n^2} = \frac{2}{(q)_{\infty}}\sum_{n \in \mathbb{Z}} \frac{(-1)^nq^{n(3n+1)/2}}{1+q^n},
\end{equation}

\noindent we have that $\mathcal{T}_1(q) = (-q)_{\infty}^3f(q^2)/(q)_{\infty}$ is a mixed mock modular form.   

For another example, set $a=-b=-r_1=1$ in \eqref{Baileytransbis} and let $r_2 \to \infty$ to obtain
\begin{equation} \label{T2}
\mathcal{T}_2(q) := \sum_{n \geq 0} \frac{(-1;q^2)_n(q;q^2)_nq^{n^2+n}}{(q^2;q^2)_n(-q)_{2n}} = \frac{2(-q^2;q^2)_{\infty}}{(q^2;q^2)_{\infty}}\sum_{n \in \mathbb{Z}} \frac{(-1)^nq^{2n^2+n}}{1+q^{2n}}.
\end{equation}
Comparing \eqref{T2} with an expression for the second order mock theta function $\mu(q)$ \cite[Entry (12.2.1)]{An-Be1},
\begin{equation*}
\mu(q) := \sum_{n \geq 0} \frac{(q;q^2)_n(-1)^nq^{n^2}}{(-q^2;q^2)_n^2} = \frac{2(q;q^2)_{\infty}}{(q^2;q^2)_{\infty}} \sum_{n \in \mathbb{Z}} \frac{q^{2n^2+n}}{1+q^{2n}}, 
\end{equation*}
we have that $\mathcal{T}_2(q) = (-q^2;q^2)_{\infty}\mu(-q)/(-q;q^2)_{\infty}$ is a mixed mock modular form.
 
It is interesting to compare what happens with Bailey's transformation to what happens with a limiting case of a transformation of Watson \cite{Ga-Ra1},

\begin{equation} \label{Watson-Whipple}
\sum_{n \geq 0} \frac{(aq/bc,d,e)_n\left(\frac{aq}{de}\right)^n}{(q,aq/b,aq/c)_n} = \frac{(aq/d,aq/e)_{\infty}}{(aq,aq/de)_{\infty}}\sum_{n \geq 0} \frac{(a)_n(1-aq^{2n})(b,c,d,e)_n(-1)^nq^{n \choose 2}(aq)^{2n}}{(q)_n(1-a)(aq/b,aq/c,aq/d,aq/e)_n (bcde)^n}.
\end{equation} 

\noindent Here there are also many ways to turn the series on the right-hand side of \eqref{Watson-Whipple} into an Appell function, but in these cases one typically obtains a genuine mock theta function.   (See \cite{Al-Br-Lo1,Br-Lo1,Br-On1,Br-On-Rh1,Ka1,Wa1}, for example.)     
 
Multi-term transformations may also be useful.  For example, consider the following three-term transformation from Ramanujan's lost notebook \cite[Entry $(3.4.7)$]{An-Be1},

\begin{equation} \label{3t}
\sum_{n \geq 1} (-a)_n(-b)_nq^n = \frac{(-a)_{\infty}}{(q,-q/b)_{\infty}} \sum_{n \in \mathbb{Z}} \frac{b^{n}q^{n+1 \choose 2}}{1+aq^n} - \sum_{n \geq 0} \frac{(ab)^{-n}q^{n^2}}{(-q/a,-q/b)_n}.
\end{equation}

\noindent Replacing $a$ by $-x$ and $b$ by $-1/x$ in \eqref{3t} we have

\begin{equation} \label{3tspec}
\mathcal{U}(x,q) := \sum_{n \geq 0} (x,1/x)_nq^n = \frac{(1-x)}{(q)_{\infty}} \sum_{n \in \mathbb{Z}} \frac{(-x)^{-n}q^{n+1 \choose 2}}{1-xq^n} - \sum_{n \geq 0} \frac{q^{n^2}}{(xq,q/x)_n}.
\end{equation}

\noindent The first term on the right-hand side of \eqref{3tspec} is, up to multiplication by an infinite product, an instance of \eqref{Appell-Lerchdef}.  As mentioned above, the second term in \eqref{3tspec} is a mock theta function when $x \neq 1$ is a root of unity.  Thus for such $x$ the series $\mathcal{U}(x,q)$ will in general be a mixed mock modular form.\footnote{When $x=\sqrt{-1}$ it is in fact a genuine mock theta function.  See \cite{Br-On-Pi-Rh1}.}   Of note is the case $x=-1$, which is the generating function for strongly unimodal sequences \cite{rhoades}.

\subsection{Bailey pairs and the Bailey chain}
A \emph{Bailey pair} relative to $a$ is a pair of sequences $(\alpha_n,\beta_n)_{n \geq 0}$ satisfying

\begin{equation} \label{pairdef}
\beta_n = \sum_{k=0}^n \frac{\alpha_k}{(q)_{n-k}(aq)_{n+k}}. 
\end{equation} 

\noindent The Bailey lemma says that if $(\alpha_n,\beta_n)$ is such a sequence, then so is $(\alpha'_n,\beta'_n)$, where

\begin{equation} \label{alphaprimedef}
\alpha'_n = \frac{(\rho_1)_n(\rho_2)_n(aq/\rho_1\rho_2)^n}{(aq/\rho_1)_n(aq/\rho_2)_n}\alpha_n
\end{equation} 

\noindent and

\begin{equation} \label{betaprimedef}
\beta'_n = \sum_{k=0}^n\frac{(\rho_1)_k(\rho_2)_k(aq/\rho_1\rho_2)_{n-k} (aq/\rho_1\rho_2)^k}{(aq/\rho_1)_n(aq/\rho_2)_n(q)_{n-k}} \beta_k.
\end{equation}

\noindent Iterating \eqref{alphaprimedef} and \eqref{betaprimedef} leads to a sequence of Bailey pairs, called the Bailey chain.  Putting \eqref{alphaprimedef} and \eqref{betaprimedef} in \eqref{pairdef} and letting $n \to \infty$ gives 

\begin{equation} \label{limitBailey}
\sum_{n \geq 0} (\rho_1)_n (\rho_2)_n (aq/\rho_1 \rho_2)^n \beta_n = \frac{(aq/\rho_1)_{\infty}(aq/\rho_2)_{\infty}}{(aq)_{\infty}(aq/\rho_1 \rho_2)_{\infty}} \sum_{n \geq 0} \frac{(\rho_1)_n (\rho_2)_n(aq/\rho_1 \rho_2)^n }{(aq/ \rho_1)_n(aq/ \rho_2)_n}\alpha_n.
\end{equation} 

\noindent For more on Bailey pairs, see \cite{AnCBMS,AnPhoenix,Wa1}.

By now, most (if not all) classical $q$-hypergeometric mock theta functions have been expressed in the literature in terms of indefinite theta series.  This is in large part thanks to work of Andrews and Hickerson \cite{andrews57,An-Hi1} on Bailey pairs wherein $\alpha_n$ contains an indefinite quadratic form.  To see how this works, let us treat Ramanujan's fifth order mock theta function,
\begin{equation*}
f_1(q) := \sum_{n \geq 0} \frac{q^{n^2+n}}{(-q)_n}.
\end{equation*}
First, Andrews \cite{andrews57} proved that $(\alpha_n,\beta_n)$ is a Bailey pair relative to $q$, where

\begin{equation} \label{5thalpha}
\alpha_n = \frac{(1-q^{2n+1})q^{n(3n+1)/2}}{1-q}\sum_{|j| \leq n} (-1)^jq^{-j^2}
\end{equation}

\noindent and

\begin{equation} \label{5thbeta}
\beta_n = \frac{1}{(-q)_n}.
\end{equation}

\noindent Next, inserting \eqref{5thalpha} and \eqref{5thbeta} into \eqref{limitBailey} with $\rho_1$, $\rho_2 \to \infty$ gives

\begin{equation*}
f_{1}(q):=\sum_{n \geq 0} \frac{q^{n^2+n}}{(-q)_n} = \frac{1}{(q)_{\infty}}\sum_{n \geq 0}\sum_{|j| \leq n} (-1)^jq^{n(5n+3)/2-j^2}(1-q^{2n+1}).
\end{equation*} 

\noindent We then have

\begin{equation} \label{f1tof}
\begin{aligned}
f_{1}(q) &=
\frac{1}{(q)_{\infty}}\left(\sum_{n \geq 0}\sum_{|j| \leq n} (-1)^jq^{n(5n+3)/2-j^2} - \sum_{n \geq 0}\sum_{|j| \leq n} (-1)^jq^{n(5n+7)/2+1-j^2}\right) \\
&= \frac{1}{(q)_{\infty}}\left(\sum_{n \geq 0}\sum_{|j| \leq n} (-1)^jq^{n(5n+3)/2-j^2} - \sum_{n < 0}\sum_{|j| \leq -n-1} (-1)^jq^{n(5n+3)/2-j^2}\right) \\
&= \frac{1}{(q)_{\infty}} \left(\left(\sum_{\substack{r, s \geq 0 \\ r \equiv s \imod{2}}} - \sum_{\substack{r, s < 0 \\ r \equiv s \imod{2} }} \right) (-1)^{\frac{r-s}{2}}q^{\frac{3}{8}r^2 + \frac{7}{4}rs+ \frac{3}{8}s^2 + \frac{3}{4}r+\frac{3}{4}s} \right) \\
&= \frac{1}{(q)_{\infty}} \left(\left(\sum_{r,s \geq 0}  - \sum_{r,s < 0}\right) (-1)^{r+s}q^{\frac{3}{2}r^2+7rs+\frac{3}{2}s^2+\frac{3}{2}r+\frac{3}{2}s} \right.  \\ &\hskip1in + \left. \left(\sum_{r,s \geq 0}  - \sum_{r,s < 0}\right) (-1)^{r+s}q^{\frac{3}{2}r^2+7rs+\frac{3}{2}s^2+\frac{13}{2}r+\frac{13}{2}s+4}\right) \\
&= \frac{1}{(q)_{\infty}} \Big(f_{3,7,3}(q^3,q^3,q) + q^4f_{3,7,3}(q^{8},q^{8},q)\Big).
\end{aligned}
\end{equation} 

\noindent In the above, we first replaced $n$ by $-n-1$ in the second sum, then set $n = (r+s)/2$ and $j=(r-s)/2$, then replaced $(r,s)$ by $(2r,2s)$ or $(2r+1,2s+1)$, and finally applied the definition in \eqref{fdef}.   This is a standard calculation \cite{Hi1}.  

We may now apply \eqref{hm} to write the last line in \eqref{f1tof} in terms of the Appell-Lerch series \eqref{Appell-Lerchdef}.  It is not necessary to write out the exact expression, only to note that using (\ref{j1})--(\ref{g}), each of the modular forms $j(x,q)$ occurring in $g_{3,7,3}(q^3, q^3, q, -1, -1)$ and $g_{3,7,3}(q^8, q^8, q, -1, -1)$ is, up to a power of $q$, either  $j(q^3,q^3) = 0$ or $j(q, q^3)=(q)_{\infty}$. The former obviously contributes nothing and the latter will cancel with the $1/(q)_{\infty}$ in \eqref{f1tof}.  Thus we have expressed $f_1(q)$, up to the addition of a modular form, as a sum of Appell-Lerch series.  This is a genuine mock theta function.  

Everything seems to have worked out perfectly.   However, if the coefficient of $n^2$ in the exponent of $q$ in \eqref{5thalpha} were \emph{not} equal to $3/2$, the result of the above calculation would \emph{not} be a mock theta function.  Indeed, we would have an indefinite theta function $f_{n,n+p,n}(x,y,q)$ with $n \neq 3$ and the modular forms $j(x,q)$ occurring in \eqref{hm} would not cancel with the $1/(q)_{\infty}$ in \eqref{f1tof}.   We would then have a mixed mock modular form.   

There are two points to make here.  First, one should not expect a given Bailey pair related to indefinite quadratic forms to yield mock theta functions, but mixed mock modular forms.   Second, simple iterations along the Bailey chain using \eqref{alphaprimedef} and \eqref{betaprimedef} naturally produce infinite families of mixed mock modular forms, but not more mock theta functions.      

To illustrate the first point, consider the Bailey pair relative to $q$,

\begin{equation} \label{Andalpha}
\alpha_n = \frac{q^{n^2}(1-q^{2n+1})}{1-q}\sum_{|2j| \leq n} (-1)^jq^{-j(3j+1)}
\end{equation}

\noindent and

\begin{equation} \label{Andbeta}
\beta_n = \frac{q^{n \choose 2}}{(q)_n(q;q^2)_n}.
\end{equation}

\noindent This pair was established by Andrews as part of his study of partitions with early conditions \cite[Eq. (4.18)]{andrewsearly}. He was interested in the generating function for $J_1(n)$, the number of partitions of $n$ such that all odd integers smaller than the largest even part appear at least twice, even parts appear without gaps and odd parts larger than the largest even part are distinct.  He showed that

\begin{equation} \label{J}
\begin{aligned}
\sum_{n \geq 0} J_1(n)q^n &= (-q;q^2)_{\infty}\sum_{n \geq 0} \frac{q^{3n^2+n}}{(q^2;q^2)_n(q^2;q^4)_n}\\
&=\frac{(-q;q^2)_{\infty}}{(q^2;q^2)_{\infty}} \sum_{n \geq 0} \sum_{|2j| \leq n} (-1)^j q^{4n^2 + 2n - j(6j+2)} (1 - q^{4n+2}),
\end{aligned}
\end{equation}

\noindent where the last line of \eqref{J} follows upon inserting \eqref{Andalpha} and \eqref{Andbeta} in \eqref{limitBailey} with $q=q^2$ and $\rho_1,\rho_2 \to \infty$.  Calculating as in \eqref{f1tof} above one finds that

\begin{equation}\label{gftof}
\begin{aligned}
\sum_{n \geq 0} J_1(n)q^n = \frac{(-q;q^2)_{\infty}}{(q^2;q^2)_{\infty}}\Big(f_{5,11,5}&(q^{12},q^{16},q^4) + q^{20}f_{5,11,5}(q^{44},q^{48},q^4)  \\  &+  q^{6}f_{5,11,5}(q^{28},q^{32},q^4) + q^{42}f_{5,11,5}(q^{60},q^{64},q^4)\Big).
\end{aligned}
\end{equation}      

\noindent Using \eqref{hm}, the right-hand side of \eqref{gftof} may be expressed in terms of modular forms and Appell-Lerch series (mock theta functions).  In this expression, the modular forms $j(x,q)$ which occur as coefficients of the Appell-Lerch series do not cancel with $(-q;q^2)_{\infty}/(q^2;q^2)_{\infty}$, and thus the generating function for $J_{1}(n)$ is in fact a mixed mock modular form.   A similar calculation occurs with Andrews' generating function for augmented tubular partitions \cite[Eq. (1.13)]{Antube}.

The second point was discussed in some detail in \cite[Section 3]{lo1}.  For example, the multisum $\mathcal{B}^{(k)}(q)$ (see \eqref{Bk}) satisfies

\begin{eqnarray} 
\mathcal{B}^{(k)}(q) &=& \frac{2}{(q)_{\infty}} \sum_{n \in \mathbb{Z}} \frac{q^{kn^2+\binom{n+1}{2}}(-1)^n}{1+q^n} 
= \frac{2i(-1)^{k}}{(q)_{\infty}} A_{2k+1}(-1, q^{-k}, q) \nonumber \\
&=& \frac{2}{(q)_{\infty}} \Biggl((-1)^k\frac{(q^{2k+1}; q^{2k+1})_{\infty}^2}{2(-q^{2k+1}; q^{2k+1})_{\infty}^2}  \label{b1k} \\ && \hskip1in + \sum_{\substack{i=1 \\ i \neq k+1}}^{2k+1} (-1)^{i+1} j(q^{k+i}, q^{2k+1}) m(-q^{k-i+1}, q^{2k+1}, q^{k+i})\Biggr). \nonumber
\end{eqnarray} 
When $k=1$ this is the mock theta function $f(q)$ but when $k \geq 2$ we have a mixed mock modular form.

For another example, we have (see (7.15), (7.20) and (7.21) in \cite{andrews57})
\begin{equation*}
\begin{aligned}
\mathcal{M}^{(k)}(q) & := \sum_{n_k \geq n_{k-1} \geq \cdots \geq n_1 \geq  0} \frac{(-q)_{n_k} q^{\binom{n_k + 1}{2} + n_{k-1}^2 + n_{k-1} + \cdots + n_1^2 + n_1}}{(q)_{n_k - n_{k-1}} \cdots (q)_{n_2 - n_1} (q^{n_1 + 1})_{n_1 + 1}} \\
& = \frac{(-q)_{\infty}}{(q)_{\infty}} \Bigl( \sum_{r, s \geq 0} - \sum_{r, s < 0} \Bigr) (-1)^{r+s}q^{kr^2 + kr + (2k+1)rs + ks^2 + ks} \\
& =  \frac{(-q)_{\infty}}{(q)_{\infty}} f_{2k, 2k+1, 2k}(q^{2k}, q^{2k}, q). 
\end{aligned}
\end{equation*}
When $k=1$ this is the tenth order mock theta function
\begin{equation*} \label{phi}
\phi(q) := \sum_{n \geq 0} \frac{q^{\binom{n+1}{2}}}{(q;q^2)_{n+1}},
\end{equation*}  
but when $k \geq 2$ we have a mixed mock modular form.

\subsection{Bailey pairs and residual partial theta identities}
Inspired by work of Andrews and Warnaar \cite{andpart,An-Wa1,Wa2}, the first author recently showed how partial theta identities often imply Bailey-type lemmas with a built-in quadratic form \cite{lovepart}.  For example, if $(\alpha_n, \beta_n)$ is a Bailey pair relative to $a$, then \cite[Eq. (1.24)]{lovepart}
\begin{equation} \label{BaileyFine}
\sum_{n \geq 0} (aq)_{2n}q^n\beta_n = \frac{1}{(q)_{\infty}}\sum_{r,n \geq 0} (-a)^nq^{3n(n+1)/2+(2n+1)r}\alpha_r
\end{equation}
and
\cite[Eq. (1.5)]{lovepart}
\begin{equation} \label{otherone}
\sum_{n \geq 0} q^n\beta_n = \frac{1}{(q,aq)_{\infty}}\sum_{r,n \geq 0} (-a)^nq^{n(n+1)/2+(2n+1)r}\alpha_r.
\end{equation}
There are around twenty such results in \cite{lovepart}, and they may be used to produce many mixed mock modular forms.  The example given in \cite{lovepart} uses the Bailey pair relative to $1$ \cite[Lemma 3.3]{An2},

\begin{equation} \label{alphares}
\alpha_n = 
\begin{cases}
(-1)^n\left(z^nq^{n \choose 2} + z^{-n}q^{n+1 \choose 2}\right),& \text{$n > 0$}, \\
1,& \text{$n=0$},
\end{cases}
\end{equation}
and
\begin{equation} \label{betares}
\beta_n = \frac{(z)_n(q/z)_n}{(q)_{2n}}.
\end{equation}
Substituting \eqref{alphares} and \eqref{betares} into \eqref{BaileyFine} and simplifying gives 
$$
\begin{aligned}
\mathcal{V}(z,q) &:=  \sum_{n \geq 0} q^n(z,q/z)_n \\ &= \frac{1}{(q)_{\infty}} \left( \sum_{n,r \geq 0} - \sum_{n,r < 0} \right) (-1)^{n+r}z^rq^{3n(n+1)/2 + r(r+1)/2 + 2nr} \\ &= \frac{1}{(q)_{\infty}} f_{3,2,1}(q^3,zq,q).
\end{aligned}
$$ 
Here \eqref{hm} does not apply, but a more general result in \cite{Hi-Mo1} does, and ensures that the above is mixed mock modular.  Of note is the case $\mathcal{V}(q,q^2)/(q;q^2)_{\infty}^2$,
$$
\frac{1}{(q;q^2)_{\infty}^2}\sum_{n \geq 0}(q;q^2)_n^2q^{2n} = \frac{1}{(q)_{\infty}(q;q^2)_{\infty}}f_{3,2,1}(q^6,q^3,q^2).
$$
The left-hand side was studied by Hikami \cite{Hik1} and dubbed a ``2nd order mock theta function", but here we see that it is really a mixed mock object. 

If we instead substitute \eqref{alphares} and \eqref{betares} into \eqref{otherone}, we obtain
\begin{equation*}
\begin{aligned}
\mathcal{W}(z,q) & := 
\sum_{n \geq 0} \frac{(z,q/z)_nq^n}{(q)_{2n}}  \\ & = \frac{1}{(q)_{\infty}^2}\left( \sum_{n,r \geq 0} - \sum_{n,r < 0} \right) (-1)^{n+r}q^{\binom{n+1}{2} + 2n + \binom{r+1}{2}}z^r \\
& = \frac{1}{(q)_{\infty}^2} f_{1,2,1}(q,zq,q), 
\end{aligned}
\end{equation*}
which is in general mixed mock modular. 
 
For a final example, consider the Bailey pair relative to $1$ \cite[L(6)]{slater},
\begin{equation} \label{Slateralpha}
\alpha_n = 
\begin{cases}
0, & \text{if $n$ is odd}, \\
1, & \text{if $n=0$}, \\
(-1)^rq^{3r^2-r}(1+q^{2r}), & \text{if $n=2r > 0$}
\end{cases}
\end{equation}
and
\begin{equation} \label{Slaterbeta}
\beta_n = \frac{1}{(q;q^2)_n(q)_n}.
\end{equation}
Substituting \eqref{Slateralpha} and \eqref{Slaterbeta} in \eqref{otherone}, we find
\begin{equation}
\begin{aligned}
\mathcal{Y}(q) &:= \sum_{n \geq 0} \frac{q^n}{(q;q^2)_n(q)_n} \\ &= \frac{1}{(q)_{\infty}^2}\left( \sum_{n,r \geq 0} - \sum_{n,r < 0} \right) (-1)^{n+r}q^{n(n+1)/2 + 4nr +3r^2+r}\\ &= \frac{1}{(q)_{\infty}^2}f_{1,4,6}(q,q^4,q).
\end{aligned}
\end{equation}

Many more examples coming from the partial theta functions in \cite{lovepart} can be found in \cite[Section 4]{mortpart}.

\section{Proof of Theorem \ref{main}}

We begin with a preliminary result from \cite{Hi-Mo1}.  Suppose $x$, $z \in \mathbb{C}^{*}:=\mathbb{C} \setminus \{ 0 \}$ with neither $z$ nor $xz$ an integral power of $q$. Two relevant properties of the sums $m(x,q,z)$ are as follows (see Proposition 2.1 and Theorem 2.3 in \cite{Hi-Mo1}).

\begin{prop} For generic $x$, $z$, $z_0$ and $z_1 \in \mathbb{C}^{*}$,

\begin{equation} \label{m1}
m(x,q,z) = x^{-1} m(x^{-1}, q, z^{-1}),
\end{equation}

\begin{equation} \label{m1.5}
m(x,q,z) = m(x, q, qz)
\end{equation}

\noindent and

\begin{equation} \label{m2}
m(x, q, z_1) = m(x, q, z_0) + \Delta(x, q, z_1, z_0),
\end{equation}

\noindent where 

\begin{equation*} \label{Deltadef}
 \Delta(x, q, z_1, z_0) := \frac{z_0 J_1^3 j(z_{1} / z_{0}, q) j(xz_{0} z_{1}, q)}{j(z_{0}, q) j(z_{1}, q) j(xz_{0}, q) j(xz_{1}, q)}.
\end{equation*}

\end{prop}

\begin{proof}[Proof of Theorem \ref{main}]

\noindent First, taking $k=2$ in (\ref{b1k}) and using (\ref{j1}), (\ref{j2}), (\ref{m1}) and (\ref{m1.5}) twice, we have

\begin{equation} \label{B12sum}
\mathcal{B}^{(2)}(q) =  \frac{4}{(q)_{\infty}} \Bigg( -j(q, q^5)m(-q, q^5, q^4)  + j(q^2, q^5)m(-q^2, q^5, q^3) + \frac{(q^5; q^5)_{\infty}^2}{(-1; q^5)_{\infty}^2} \Bigg).
\end{equation}
Next, equations (4.45) and (4.46) in \cite{Hi-Mo1} state that
\begin{equation*}\label{Xm}
X(q) = 2m(-q^2,q^5,q^4) - \frac{J_{3,10}J_{5,10}}{J_{1,5}}
\end{equation*}
and
\begin{equation*} \label{chim}
\chi(q) = 2m(-q,q^5,q^2) + q\frac{J_{1,10}J_{5,10}}{J_{2,5}}, 
\end{equation*}
and applying \eqref{m2} to each of these gives
\begin{equation}\label{XmDelta}
X(q) = 2m(-q^2,q^5,q^3) + 2\Delta(-q,q^5,q^4,q^3) - \frac{J_{3,10}J_{5,10}}{J_{1,5}}
\end{equation}
and
\begin{equation} \label{chimDelta}
\chi(q) = 2m(-q,q^5,q^4) + 2\Delta(-q,q^5,q^2,q^4) + q\frac{J_{1,10}J_{5,10}}{J_{2,5}}. 
\end{equation}
Using \eqref{XmDelta} and \eqref{chimDelta} with (\ref{B12sum}), we have 

$$
\begin{aligned}
\mathcal{B}^{(2)}(q) & + \frac{2}{(q^2, q^3; q^5)_{\infty}} \chi(q) - \frac{2}{(q, q^4; q^5)_{\infty}} X(q) \\ & =  \frac{2q j(q, q^5) J_{1,10} J_{5,10}}{(q)_{\infty} J_{2,5}} + \frac{4 j(q, q^5) \Delta(-q, q^5, q^2, q^4)}{(q)_{\infty}} 
 + \frac{2j(q^2, q^5) J_{3,10} J_{5,10}}{(q)_{\infty} J_{1,5}} \\ & - \frac{4j(q^2, q^5) \Delta(-q^2, q^5, q^4, q^3)}{(q)_{\infty}} 
 + \frac{(q^5;q^5)_{\infty}^2}{(-q^5;q^5)_{\infty}^2(q)_{\infty}}.
\end{aligned}
$$

\noindent  To finish the proof of \eqref{b1chix}, we need to show that the right-hand side is equal to $-(q)_{\infty}/(-q)_{\infty}^2$.   But this is just an identity between modular forms, which may be verified with a finite computation.   We have carried this out using Garvan's MAPLE program available at 

\begin{center}
\url{http://www.math.ufl.edu/~fgarvan/qmaple/theta-supplement}. 
\end{center}

\end{proof}

\section{Concluding Remarks}
Each of the ideas in Section $3$ is well worth pursuing.  First, the transformations of Bailey are just two of the many transformations which arise from a change-of-base in Bailey pairs  \cite{Be-Wa1, Br-Is-St1, Ja-Ve1}.  One could investigate all possible Appell functions occurring as specializations of such transformations.   Second, one could look at the proofs of expressions for the classical second, third, fifth, sixth, eighth, and tenth order mock theta functions in terms of Appell-Lerch series and/or indefinite theta series.   Behind each such proof lies a Bailey pair, and iterating along the Bailey chain would then embed each classical mock theta function in an infinite family of $q$-hypergeometric mixed mock modular forms, just as the mock theta functions $f(q)$ and $\phi(q)$ are the base cases of the families $\mathcal{B}^{(k)}$ and $\mathcal{M}^{(k)}$, respectively.   Finally, one could consider the Bailey-type lemmas in \cite{lovepart} together with all of the Bailey pairs occurring in Slater's list \cite{slater} and see what kinds of mixed mock modular forms arise. 

\section*{Acknowledgements}
The authors would like to thank Kathrin Bringmann for her helpful conversations. The second author would like to thank the Institut des Hautes {\'E}tudes Scientifiques for their support during the preparation of this paper.

\end{document}